\documentclass[10pt]{amsart}
\usepackage{amssymb}

\newtheorem{theorem}{Theorem}

\newtheorem{claim}{Claim}
\newtheorem{lemma}{Lemma}
\theoremstyle{definition}
\newtheorem{definition}[theorem]{Definition}
\newtheorem{example}{Example}
\newtheorem{remark}{Remark}

\def\MARU#1{{\ooalign{\hfil#1\/\hfil\crcr\raise.167ex\hbox{\mathhexbox20D}}}}

\newenvironment{namelist}[1]{%
\begin{list}{}
  {
   \settowidth{\labelwidth}{#1}
   \setlength{\leftmargin}{2.5\labelwidth}}
}{%
\end{list}}

\begin{document}

\bigskip

\title{Galois Group at Galois Point for genus-one Curve}
\author{Mitsunori Kanazawa and Hisao Yoshihara}
\address{Graduate School of Science and Technology,\\
 Niigata University Niigata 950-2181, Japan and \\
 Department of Mathematics, Faculty of Science, \\ 
Niigata University, Niigata 950-2181, Japan}
\email{mi\_{}kana@js3.so-net.ne.jp and yosihara@math.sc.niigata-u.ac.jp}
\maketitle

\maketitle
\begin{abstract}
We show all the possible structures of finite subgroups of the automorphism groups of elliptic curves. 
Using the result, we determine every Galois group $G$ at outer Galois point for genus-one curve with singular points. 
In particular, if $G$ is abelian, then its order is at most nine.  However if not so, the order is unbounded. 
Furthermore, we give every defining equation of the curve 
with the Galois point in the case where $G$ is abelian.  
\end{abstract}
\bigskip

\section{Introduction}
The purposes of this article are (1) to determine every possible Galois group $G$ of a Galois point for a genus-one curve with a singular point 
and (2) to give the defining equation of the curve when $G$ is abelian. 
First, let us recall the definition of (outer) Galois point (\cite{my}, \cite{y1}). 

Let $k$ be the ground field of our discussion.  We assume it to be an algebraically closed field of characteristic zero. 
Let $C$ be an irreducible projective plane curve of degree $d$ ($\geq 3$) and $k(C)$ 
the function field. Let $P$ be a point in the plane $\mathbb P^2 \setminus C$ and consider the projection 
from $P$ to $\mathbb P^1$, 
$\pi_P : \mathbb P^2 \dasharrow \mathbb P^1$. Restricting $\pi_P$ to $C$, we get a surjective morphism 
$\bar \pi_P : C \longrightarrow \mathbb P^1$, which induces a finite extension of fields 
$\bar \pi_P^* : k(\mathbb P^1) \hookrightarrow k(C)$. 
If the extension is Galois, we call $P$ a (outer) Galois point for $C$. 
Let $G=G_P$ be the Galois group $\mathrm{Gal}(k(C)/\bar \pi_P^*(k(\mathbb P^1)))$. 
We call $G$ the Galois group at $P$. 
By definition each element of $G$ induces a birational transformation of $C$ over the projective line $\mathbb{P}^1$. 
If $C$ is smooth, then the element is an automorphism of $C$.  Moreover, if $d \geq 4$, then it can be extended to a projective 
transformation of $\mathbb{P}^2$ and $G$ turns out to be a cyclic group (\cite{y1}).
 However, in case $C$ has a singular point, several new phenomena occur, for examples, the group is not necessarily cyclic, 
and the element of $G$ cannot necessarily be extended to a birational transformation of $\mathbb P^2$ (cf. \cite{y2}).  
It seems interesting to determine Galois group when $C$ has a singular point (cf. \cite{m}). 
 
In this article we assume $k=\mathbb C$; the field of complex numbers. 
By a {\it genus-one curve} we mean it is an irreducible plane curve whose smooth model has the genus one.
In the first half we study finite subgroups of the automorphism groups of smooth genus-one curves, i.e., elliptic curves, and then 
determine the Galois groups of the Galois points for such curves. 
The results are stated in Section 2.  
In the latter half we give the defining equations of the curves  
when the groups are abelian. 
Similar study for space elliptic curves and abelian surfaces have bean done in \cite{y3} and \cite{y4}, respectively.

Note that if the characteristic of the ground field $k$ is positive, then many new phenomena occur and there exist lots of different facts.   
For the recent development of positive characteristic case, see \cite{f}. 

We use the following notation and convention, where $n$ is a positive integer:

\begin{namelist}{3}
\item[$\cdot$]$E$ : an elliptic curve
\item[$\cdot$]$\mathcal L=\mathbb Z+\mathbb Z \omega$ : the lattice defining $E$, where $\Im \omega >0$
\item[$\cdot$]$A(E)$ : the automorphism group of $E$ as a variety
\item[$\cdot$]$e_n:=\exp(2\pi \sqrt{-1}/n)$ 
\item[$\cdot$]$Z_n:=\mathbb Z/n\mathbb Z$ 
\item[$\cdot$]$D_n$ : the dihedral group of order $2n$
\item[$\cdot$]$|G|$ : the order of a finite group $G$
\item[$\cdot$]$\langle \sigma_1, \cdots, \sigma_n \rangle$ : the subgroup of some group generated by $\sigma_1, \cdots, \sigma_n$ 
\end{namelist}

\medskip

\medskip

\section{Statement of Results}
Let $E$ be the elliptic curve $\mathbb C/\mathcal L$ and $A(E)$ the automorphism group of $E$ as a variety. 
For an element $\sigma \in A(E)$, there exists a lift $\widetilde{\sigma} : \mathbb C \longrightarrow \mathbb C$ such that 
 $\widetilde{\sigma}z=\alpha z+\beta$, where 
$z \in \mathbb C$, $\alpha=\alpha(\sigma)$ and $\beta=\beta(\sigma)$ are constants in  $\mathbb C$. 
Then we have $\alpha \mathcal L=\mathcal L$. 
Note that $\alpha^n=1$, where  $n=1,2,3,4$ or $6$. 
Fixing the representation, we define 
\[
T(E)=\{\ \sigma \in A(E) \ | \ \alpha(\sigma)=1 \  \} 
\]
and
\[ 
A(E)_0=\{\ \alpha(\sigma) \in \mathbb C \ | \ \beta(\sigma)=0 \  \}. 
\]
It is well-known that $A(E)_0$ is a cyclic group of order $2, 4$ or $6$. 
We have the following exact sequence of the groups:
\[
1 \longrightarrow T(E) \stackrel{i}{\longrightarrow} A(E) \stackrel{\alpha}{\longrightarrow} A(E)_0 \longrightarrow 1,  
\] 
where $i$ is the natural injection and $\alpha$ is defined as above. 
Since $A(E)_0$ can be viewed as a subgroup of $A(E)$, the $A(E)$ is a semi-direct product $T(E) \rtimes A(E)_0$. 
In case $G$ is a subgroup of $A(E)$, put $G_T=G \cap T(E)$ and $G_0=\alpha(G)$. Restricting the sequence onto $G$, we get the following exact sequence:  
\[
1 \longrightarrow G_T \longrightarrow G \longrightarrow G_0 \longrightarrow 1.  
\]
The group $G$ is a semi-direct product $G_T \rtimes G_0$. This is a well-known fact, in this article we will find more 
detailed property, i.e., the possibility of $|G_T|$ and the way of action of $G_0$ onto $G_T$. 
This is a classical question, but there may be no 
appropriate reference, so we present here for the completeness. 
In order to state the theorems we prepare the following terms. 

\begin{definition}\label{3}
A finite group $G$ is called a {\it bidihedral group} if it is generated by the elements $a,b$ and $c$ satisfying the following conditions (1), (2) and (3):

\begin{enumerate}
\item The orders of $a$, $b$ and $c$ are $2$, $m$ and $n$ respectively. 
\item $ n \geq m \geq 2$ and $n \geq 3$
\item $ aba=b^{-1}, \ aca=c^{-1}$ and $bc=cb $ 
\end{enumerate}

We denote this group by $BD_{mn}$. 
\end{definition}

\begin{definition}\label{5}
A finite non-abelian group $G$ of order $m^2kl$ is called an {\it exceptional elliptic group} if it satisfies the  
 following conditions (1), (2) and (3):
\begin{enumerate}
\item[(1)] $G$ is the semi-direct product $H \rtimes K$ with some action of $K$ onto $H$ , where $K \cong Z_l$ and $H$ is a  
normal subgroup of $G$ such that $H \cong Z_k$ or $Z_m \oplus Z_{mk}$. 
\item[(2)] $l=3,4$ or $6$, and $m$ is any positive integer.  
\item[(3)] The positive integer $k$ is a factor of $h^2+\varepsilon h+1$, where $h$ is an integer and $\varepsilon=0$ (resp. $1$) 
corresponding to $l=4$ (resp. $3$ or $6$). 
\end{enumerate}
We denote these groups by $E(k, l)$ and $E(m,k,l)$ if $H$ has one and two generators, respectively. 
\end{definition}

Concerning the possibility of $k$ in (3) of Definition \ref{5} we note the following. 

\begin{remark}
In case $k \ne 1$, each prime factor $p$ of $k$ is as follows:  
 \begin{enumerate}
 \item If $l=3$ or $6$, then $p=3$ or $p \equiv 1 \pmod{3}$. 
 \item If $l=4$, then $p=2$ or $p \equiv 1 \pmod{4}$. 
 \end{enumerate}
\end{remark}

Using the above definitions, we can state all the possible structures of the finite subgroups $G$ as follows.

\begin{theorem}\label{20}
A finite group $G$ can be a subgroup of $A(E)$ for some elliptic curve $E$ if and only if $G$ is isomorphic to one of the 
following{\rm :}  

\begin{enumerate}
\item[(1)] abelian case: 
\begin{enumerate}
\item[(1.1)] $Z_m$ {\rm (}$m \geq 1${\rm)} or $Z_m \oplus Z_{mn} $ \ {\rm(}$m \geq 2,\ n \geq 1${\rm)} 
\item[(1.2)] $Z_2$, \ ${Z_2}^{\oplus2}$,\ ${Z_2}^{\oplus3}$,\ $Z_3$, \ ${Z_3}^{\oplus2}$, \ 
$Z_4$, \ $Z_2\oplus Z_4$ \ or \ $Z_6$ 
\end{enumerate}
\item[(2)] non-abelian case:  
\begin{enumerate}
\item[(2.1)] $D_n$ or $BD_{mn}$ $(n \geq 3)$  
\item[(2.2)] $E(k, l)$ or $E(m,k,l)$ 
\end{enumerate}
\end{enumerate} 
The cases {\rm (1.1), (1.2), (2.1)} and {\rm (2.1)} appear in the cases where $|G_0|=1, \ |G_0|>1, \ |G_0|=2$ and $|G_0|>2$ 
respectively. 

\end{theorem}

\begin{remark}
In the case (2) of Theorem \ref{20}, even if $|G|$ is given, there exist many distinct groups. 
For example, let $|G|=1300m^2$, where $m$ is any positive integer. Notice $1300=4\cdot325=4\cdot5^2\cdot13$. 
We consider automorphism groups on the elliptic curve $E=\mathbb C/\mathbb Z+\mathbb Z i$, where $i=e_4$. 
Let $\beta=1/5m $ and $\beta'=(-5+i)/65m$. Then we have
\[
i(\beta, \beta')=(\beta, \beta')M, \ \mathrm{where} \ M=\begin{pmatrix}5 & -2 \\ 13 & -5 \end{pmatrix} \in SL_2(\mathbb Z). 
\]
Hence $G_T \cong Z_{5m} \oplus Z_{65m}$ and $G \cong E(5m,13,4)$. 
On the other hand, let $\beta=1/m$ and $\beta'=(57+i)/325m$. 
Then we have 
\[
i(\beta, \beta')=(\beta, \beta')M, \ \mathrm{where} \ M=\begin{pmatrix}-57 & -10 \\ 325 & 57 \end{pmatrix} \in SL_2(\mathbb Z). 
\] 
Hence $G_T \cong Z_{m} \oplus Z_{325m}$ and $G \cong E(m,325,4)$. 
\end{remark}

The following is the key lemma to obtain the main theorem. 

\begin{lemma}\label{1}
A finite subgroup $G$ of $A(E)$ can be a Galois group at Galois point for genus-one curve $C$ if and only if 
 $|G| \geq 3$ and $|G_0| \ne 1$.
\end{lemma}

Thus the main theorem is stated as follows: 

\begin{theorem}\label{4}
A finite group $G$ can be the Galois group at a Galois point for some genus-one curve if and only if $G$ is isomorphic to one of the 
following{\rm :}  

\begin{enumerate}
\item[(1)] abelian case{\rm :} 

 ${Z_2}^{\oplus2}$,\ ${Z_2}^{\oplus3}$,\ $Z_3$, \ ${Z_3}^{\oplus2}$, \ 
$Z_4$, \ $Z_2\oplus Z_4$ \ or \ $Z_6$ 

\item[(2)] non-abelian case{\rm :}   
\begin{enumerate}
\item[(2,1)] $D_n$ or $BD_{mn}$ $(n \geq 3)$
\item[(2.2)]
$E(k,l)$ or $E(m,k,l)$ 
\end{enumerate}

\end{enumerate}
\end{theorem}

Notice that (1) if $C$ is smooth, i.e., an elliptic curve, then $G \cong Z_3$, 
(2) only a finitely many groups can appear if they are abelian. Indeed, their orders are 
at most $9$.

\section{Proofs}
First we prove Theorem \ref{20}. 
If $|G_0|=1$, then $G = G_T$. Since $G_T$ is generated by one or two elements,  
the assertion (1.1) is clear. 
 
Hereafter we assume $G_0=\langle e_l \rangle$, where $l=2, 3, 4$ or $6$. 
Take $\sigma \in A(E)$ such that $\widetilde{\sigma}(z)=e_lz$ and let $G_T= \langle \tau \rangle$ or $\langle \tau, \tau' \rangle$,  
where the latter is the case when $G_T$ has two generators.   
Take the representations $\widetilde{\sigma}(z)=e_lz$, $\widetilde{\tau}(z)=z+\beta$ and $\widetilde{{\tau}'}(z)=z+{\beta}'$ on the universal 
covering $\mathbb C$. 
We have the equality $e_l^2+\varepsilon e_l+1=0$ where $\varepsilon=1,0,-1$ 
corresponding to $l=3,4,6$ respectively. 
Here we note that $\mathbb C/\mathbb Z+\mathbb Ze_3$ is isomorphic to $\mathbb C/\mathbb Z+\mathbb Ze_6$.  

\medskip

The following lemma may be well-known (cf. \cite[Ch.III, \S 10]{s}), 

\begin{lemma}\label{6}
If the elliptic curve $E$ has an automorphism of order $l$ $(l=3, 4, 6)$ with a fixed point, then $E$ is isomorphic 
to $\mathbb C/\mathbb Z+\mathbb Ze_l$.  
\end{lemma}

We prove the theorem by distributing the cases (I) $G$ is abelian and (II) $G$ is not abelian. 

\bigskip

(I) The case where $G$ is abelian. 

\begin{claim}\label{21}
If $l=2,\ 3,\ 4$ and $6$, then $2\beta,\ 3\beta, \ 2\beta$ and $\beta$ belong to $\mathcal L$, respectively. 
\end{claim} 

\begin{proof} 
Since $\sigma \tau=\tau \sigma$, we have $(e_l-1)\beta \in \mathcal L$. 
If $l=2$, then $e_l=-1$, hence $2\beta \in \mathcal L$. 
On the other hand, if $l \ne 2$, then $e_l^2+\varepsilon e_l+1=0$ where $\varepsilon=1,0,-1$. 
Using the property $e_l \mathcal L=\mathcal L$, we infer readily that $3\beta, 2\beta$ and  $\beta$ belong to $\mathcal L$ in the case where 
$l=3,4$ and $6$ respectively. 
\end{proof}

Using this claim we conclude the following assertions. 

\begin{enumerate}
\item[(i)] If $l=2$, then $2\beta$ (and $2{\beta}'$) $\in \mathcal L$. Therefore we have $G \cong Z_2, \ Z_2^{\oplus 2}$ or $Z_2^{\oplus 3}$ 
corresponding to $G_T \cong id, \ Z_2$ and $Z_2^{\oplus 2}$ respectively. 
\item[(ii)] If $l=3$, then we have $3\beta \in \mathcal L$. By Lemma \ref{6} we can put $\beta=(a+be_l)/3$. Since $(e_l-1)\beta \in \mathcal L$, we have $(a, \ b)=(1,\ 2)$ or $(2, \ 1)$, which means that $G_T$ is generated by one 
element. Therefore we have 
$G \cong Z_3$ or $Z_3^{\oplus 2}$. 
\item[(iii)] If $l=4$, then by Lemma \ref{6} we can put $\beta=(1+e_l)/2$, similarly we infer $G \cong Z_4$ or $Z_2 \oplus Z_4$. 
\item[(iv)] If $l=6$, then $\beta \in \mathcal L$, i.e., $\tau$ is identity. Therefore we have $G \cong Z_6$.  
\end{enumerate}

Conversely, observing each step of the proof, we see that the group appeared above can be a finite subgroup of $A(E)$ for some $E$.

\medskip
\medskip

(II) The case where $G$ is not abelian.

\medskip
 
First, we treat the case where  $G_0=\langle-1\rangle$. 
Then, clearly we have $\sigma \tau \sigma^{-1}=\tau^{-1}$. Hence $G_T$ is a dihedral or bidihedral groups according to that $G_T$ is generated by 
one or two elements, respectively. 

Next, we treat the case where $G_0=\langle e_l \rangle$ ($l=3, 4$ or $6$). 

\smallskip

If $G_T$ is generated by one element of order $k$, then $k$ satisfies some condition as follows. 
We can express $\sigma \tau \sigma^{-1}=\tau^h$ for some integer $h$. 
Let $\beta=(a+be_l)/k$, where $gcd(a,b,k)=1$. 
Then, we have $e_l\beta \equiv h\beta \pmod{\mathcal L}$. This is equivalent to  
\[
\left\{
\begin{array}{ccc}
ha+b & \equiv & 0 \pmod{k} \\
-a+(h+\varepsilon)b & \equiv & 0 \pmod{k}. 
 \end{array} \right. 
\]
From these equations we infer readily that $k$ is a factor of $h^2+\varepsilon h+1$. 
Conversely, suppose a factor $k$ of $h^2+\varepsilon h+1$ is given. 
Express $h^2+\varepsilon h+1=kk'$ and put $\beta=((h+\varepsilon)+e_l)/k$. Then, $\beta \pmod{\mathcal L}$ has order $k$. 
We have $e_l\beta=(e_lh-1)/k$ and $h\beta=k'+(he_l-1)/k$. 
Thus there exists a finite subgroup $G$ of $A(E)$ such that $G$ is isomorphic to 
$E(k, l)$. 

Finally we prove the remainder case. 
Assume $G_T$ needs two generators $\tau$ and $\tau'$ such that
 ord$(\tau)=m$, ord$(\tau')=m'$, where $m'=mk$. 
 So $G_T$ is isomorphic to $Z_m \oplus Z_{m'}$ as $\mathbb Z$-modules. 
 Similarly as above, $k$ must satisfy some condition. In order to find it, we examine the action of $G_0$ onto $G_T$.  
Let the lifts of $\tau$ and $\tau'$ on $\mathbb C$ be $\widetilde{\tau}(z)=z+\beta$ and $\widetilde{\tau'}(z)=z+\beta'$ such that 
 $m\beta \in \mathcal L$ and $m'\beta' \in \mathcal L$, respectively. 
The actions of $G_0$ on $G_T$ are given by  $\widetilde{\sigma \tau \sigma^{-1}}(z)=z+e_l \beta$ and $\widetilde{\sigma \tau' \sigma^{-1}}(z)=z+e_l \beta'$. 
Thus the actions are represented by a matrix $A_l$ such that $e_l(\beta, \beta')=(\beta, \beta')A_l$, where $A_l \in \mathrm{Aut}(Z_m \oplus Z_{m'})$. 
This defines the structure of the semi-direct product $G=G_T \rtimes G_0$. 

\begin{lemma}\label{7} 
If $G_T$ needs two generators $\tau$ and $\tau'$, then $\beta$ and $\beta'$ are linearly independent over $\mathbb R$.
\end{lemma}

\begin{proof}
Suppose the contrary. 
Then, there exists $c \in \mathbb R$ such that $c \ne 0$ and $\beta'=c\beta$. 
Let $L$ be the line $\{\ a\beta+a'\beta' \ | \ a,\ a' \in \mathbb R \   \}$ in $\mathbb C$. 
Put $L_G=\{\ \lambda \ | \ \lambda \in L \ \mathrm{and} \ \tau \in G_T, \ \mathrm{where}\ \widetilde{\tau}(z)=z+\lambda  \}$. 
Since $|G_T|$ is finite, there exists $\lambda_0 \in L_G$ ($\lambda_0 \ne 0$) such that the norm of $\lambda_0$ is minimum in $L_G$. 
Then, $G_T$ is generated by $\lambda_0$, which is a contradiction.
\end{proof}

Let us determine the matrix $A_l$. 
Since $\beta$ and $\beta'$ are linearly independent over $\mathbb R$, the group $\mathcal L'=\mathbb Z \beta + \mathbb Z \beta'$ makes a lattice 
defining the elliptic curve $E'=\mathbb C/\mathcal L'$. 
 Since $G_0$ acts on $G_T$, the elliptic curve $E'$ has the automorphism of order $l$ with fixed points. By Lemma \ref{6}, $E'$ is isomorphic to 
the original $E=\mathbb C/\mathcal L$. 
Hence we have   
\[
(\beta,\ \beta')=\lambda(1,\ e_l)M , \ \mathrm{where} \ M=\begin{pmatrix}p & r \\ q & s \end{pmatrix} \in GL_2(\mathbb Z) 
\ \mathrm{and} \ \lambda \in \mathbb C. \eqno{(a)} 
\]
Using the basis $(1,\ e_l)$, we can represent the action of $\sigma$ as $\sigma(\beta, \beta')= (e_l \beta, e_l \beta')=(\beta, \beta')A_l$.   
In this expression $A_l$ is given by  
\[
A_l=M^{-1} B_l M, \ \mathrm{where} \ B_l=\begin{pmatrix}0 & -1 \\ 1 & -\varepsilon \end{pmatrix}
\]
such that $\varepsilon = 1, 0$ and $-1$, corresponding to $l=3,4$ and $6$ respectively. 

\begin{remark}
If $k > 1$, then $\mathrm{Aut}(Z_m \oplus Z_{m'}) \cong \mathrm{Aut}(Z_m) \oplus \mathrm{Aut}(Z_{m'})$, hence $A_l$ is diagonal. 
\end{remark}

\begin{example}\label{8}
Let us examine $A_l$ by an example in the case where $l=4$. 
Let $\mathcal L=\mathbb Z+\mathbb Zi$, where $i=e_4$, $\beta=(2+i)/5$ and $\beta'=(3+i)/10$. Clearly we have  
\[
(i\beta,\ i\beta') \equiv (\beta, \ \beta') \begin{pmatrix}2 & 0 \\ 0 & 3 \end{pmatrix} \pmod{\mathcal L}. 
\eqno{(b)}   
\]
On the other hand, from the equation $(\beta, \ \beta')=\lambda (1, \ i)M$ we infer that 
\[
M= \pm \begin{pmatrix}1 & 1 \\ 3 & 2 \end{pmatrix} \ \mathrm{or} \ \ \pm \begin{pmatrix}3 & 2 \\ -1 & -1 \end{pmatrix}.   
\]
Hence we have 
\[
A_4= \begin{pmatrix}7 & 5 \\ -10 & -7 \end{pmatrix}. 
\]
Since $5\beta \equiv 0 \pmod{\mathcal L}$ and $10\beta' \equiv 0 \pmod{\mathcal L}$, this action is the same as the above $(b)$. 
\end{example}

\begin{remark}\label{9}
Put $\varepsilon =1, 0, -1$ corresponding to $l=3,4,6$ respectively.
Then $A_l$ can be expressed as 
\[
\pm \begin{pmatrix}-pr+\varepsilon qr-qs & -r^2+\varepsilon sr-s^2 \\ p^2-\varepsilon pq+q^2 & pr-\varepsilon ps+qs \end{pmatrix},   
\]
where $ps-qr=\pm1$. 
\end{remark}

\bigskip

We take $-\beta$ instead of $\beta$ if necessary, then we can assume $\Im (\beta'/\beta) > 0$, hence $M$ belongs to $SL_2(\mathbb Z)$ in $(a)$.  
By the relation $(a)$ we have  
\[
m'(\beta, \beta')M^{-1}=m'\lambda (1, e_l).
\] 
Since $m'\beta$ and $m'\beta'$ belong to $\mathcal L$, we see that $\lambda$ can be expressed as $d(a+be_l)/m'$, where $a, b$ and $d$ are 
integers such that $gcd(a,b)=1$ and $gcd(d,m')=1$. 
Therefore, by $(a)$ we can express $\beta$ and $\beta'$ as 
\[
\beta=\frac{d}{mk}  \{(ap-bq)+(bp-\varepsilon bq +aq)e_l  \} \eqno{(c)}
\]
and
\[\beta'=\frac{d}{mk} \{ (ar-bs)+(br-\varepsilon bs+as)e_l \}, \eqno{(d)}
\]
where $\varepsilon=1, 0, -1$ corresponding to $l=3,4,6$ respectively. 
Put $\bar{\beta} \equiv \beta$ (mod $\mathcal L$) and $\bar{\beta'} \equiv \beta'$ (mod $\mathcal L$). 
Then, $\mathrm{ord} \bar{\beta}=m$ and $\mathrm{ord} \bar{\beta'}=m'$ if and only if the following conditions are satisfied 
under the assumptions that $gcd(a, b)=1$ and $ps-qr=1$: 
\[
\left\{
\begin{array}{ccc}
 gcd(ap-bq,\ bp-\varepsilon bq+aq,\ mk) & = & k \\
 gcd(ar-bs,\ br-\varepsilon bs+as,\ mk) & = & 1 \\
 \end{array} \right. \eqno{(e)} 
\]

Now we proceed with the proof.  
If such a $G$ exists, then take the generators $\beta$ and $\beta'$ of $G_T$ as above. 
So we have $g=m^2k$ and $m'=mk$. 
By the condition $(e)$ we have the relations $ap-bq \equiv 0 \pmod{k},\ bp-\varepsilon bq+aq \equiv 0 \pmod{k}$. 
These relations imply $(a^2-\varepsilon ab+b^2)p \equiv 0 \pmod{k}$ and $(a^2-\varepsilon ab+b^2)q \equiv 0 \pmod{k}$. Since $gcd(p,q)=1$,  
$k$ is a factor of $a^2-\varepsilon ab+b^2$. Moreover, since $gcd(a, b)=1$, we can find an integer $h$ such that $k$ 
is a factor of $h^2+\varepsilon h+1$. 

Conversely, suppose $g$ satisfies the condition. 
Then, let $k$ be a factor of $h^2+\varepsilon h+1$. It is not difficult to see that 
$k$ can be expressed as $a^2-\varepsilon ab+b^2$ for some integers $a$ and $b$ such that $(a, b)=1$. Then, by the 
following lemma, the proof is complete.  

\begin{lemma}\label{10} 
Suppose the positive integer $k$ is expressed as $a^2-\varepsilon ab+b^2$ for some integers $a,\ b$, where $gcd(a, b)=1$. 
Then, there exists a finit subgroup $G$ of $A(E)$ for some elliptic curve $E$ such that $|G_T|=m^2k$.   
\end{lemma}

\begin{proof}
Assume such $\{a, b, k \}$ are given. Then, put $m'=mk$ and let $d$ be any positive 
integer satisfying $gcd(d, m')=1$. Then let $\beta$ and $\beta'$ be given by (c) and (d), respectively. 
Put $p=a-\varepsilon b, \ q=-b$ in $(c)$, then there exist $r$ and $s$ satisfying $sa+(r-\varepsilon s)b=1$. 
Clearly we have $M=\begin{pmatrix}p & r \\ q & s \end{pmatrix} \in SL_2(\mathbb Z) $. 
Further, put $\lambda=d(a+be_l)/m'$. 
We have 
\[
\beta=\frac{d}{m} \ \ \mathrm{and} \ \ \beta'= \frac{d}{mk}(ar-bs+e_l), 
\]
hence $\beta$ and $\beta'$ make the lattice $\mathcal L'$. 
Since 
\[
e(\beta, \beta')=\lambda (e, e^2)M = \lambda (1, e) \begin{pmatrix}0 & -1 \\ 1 & -\varepsilon \end{pmatrix}=(\beta, \beta')M', 
\]
where 
\[
M'=M^{-1} \begin{pmatrix}0 & -1 \\ 1 & -\varepsilon \end{pmatrix} M \in SL_2(\mathbb Z),  
\]
the elliptic curve $\mathbb C/\mathcal L'$ 
has the automorphism of order $l$. 
By the definitions of $\beta$ and $\beta'$ we see that the order of $G_T=\langle \beta, \beta' \rangle$ on $E$  
is $m^2k$.   
\end{proof}

\medskip

Now we proceed with the proof of Lemma \ref{1}. 
Assume the ``if part". 
Since $|G_0| \ne 1$, $G$ has a fixed point and hence $E/G$ is a rational curve. Let $p:E \longrightarrow E/G \cong \mathbb P^1$ be the quotient morphism 
 and $D$ the pullback divisor of a 
point of $\mathbb P^1$. The divisor $D$ is very ample and let $\varphi=\varphi_D$ be the embedding given by a basis of $\mathrm{H}^0(E, \mathcal O(D))$. 
Take a set of basis $\{s_0, s_1, \ldots, s_{n-1} \}$, where $n=|G|$, such that $s_0$ and $s_1$ come from $\mathrm{H}^0(\mathbb P^1, \mathcal O(P))$. 
Let $(X_0, X_1, \ldots, X_{n-1})$ be a set of homogeneous coordinates of $\mathbb P^{n-1}$ and $\pi_V$ be the projection with the center 
$V:X_0=X_1=0$. Then, we have $V \cap \varphi(E)=\emptyset$ and $p=\pi_V \cdot \varphi$.

\begin{claim}\label{2}
There exists a linear subvariety $W$ of $V$ with $\dim W=\dim V-1$ such that the projection $\pi_W$ with the center $W$ gives the 
birational morphism from $\varphi(E)$ onto its image.
\end{claim}

\begin{proof}  
Take a point $Q \in \varphi(E)$ and let $C_Q$ be a cone with the center $Q$, i.e., $C_Q= \bigcup \{\ \ell_{QQ'} \ | \ Q' \in \varphi(E) \}$, where $\ell_{QQ'}$ is the line 
passing through two points $Q$ and $Q'$. Taking another point $Q$ if necessary, we may assume 
$V$ is not contained in $C_Q$. Let $T_Q$ be the tangent line to $\varphi(E)$ at $Q$ and $R$ be a point in $V \setminus (C_Q \cup T_Q)$. 
Consider the projection $\pi_R : \mathbb P^{n-1} \dashrightarrow \mathbb P^{n-2}$ with the center $R$. Then the restriction $\pi_R$ to $\varphi(E)$ 
does not ramify at $Q$ and the fiber of 
$\pi_R(\varphi(E))$ consists of one point, i.e., it is a birational morphism from $\varphi(E)$ onto its image. Repeating these procedures, we get 
the required projection to the plane.    
\end{proof}

Let $C$ be the curve $\pi_W(\varphi(E))$. It is an irreducible plane curve with the outer Galois point $\pi_W(V)$. It is easy to see the Galois group is isomorphic to 
$G$. Conversely, assume $G$ is the Galois group.  Let $E$ be the normalization of $C$. Then $G$ turns out to be an automorphism group of $E$, and by definition we have that $E/G$ is isomorphic to $\mathbb P^1$. Since $\deg C \geq 3$, we see easily that 
$|G_0| \ne 1$ and $|G| \geq 3$. Thus we complete the proof. 

\medskip

Combining Theorem \ref{20} and Lemma \ref{1}, we infer readily Theorem \ref{4}. 
Thus we complete all the proofs.

\section{Defining Equations}
In this section we give the defining equations of genus-one curves with the Galois points whose Galois groups are abelian. 
We use the following notation. 

\begin{namelist}{3}
\item[$\cdot$]$E$ : the smooth elliptic curve defined by $y^2=x^3+px+q$ or $y^2=x(x-1)(x-b), \ b \ne 0, \pm 1$ 
\item[$\cdot$]$\mathbb C(x,y)$ : the function field of $E$ 
\item[$\cdot$]$(X:Y:Z)$ : homogeneous coordinates such that $x=X/Z$ and $y=Y/Z$ 
\item[$\cdot$]$K^G$ : the invariant subfield of a field $K$ by a group $G$ action
\item[$\cdot$]$\mathrm{div}(t)$ : the divisor of $t \in \mathbb C(x,y)$ on $E$ 
\end{namelist}

\medskip

The following Remark is useful to find the examples. 

\begin{remark}\label{12}
Let $G$ be the group in Theorem \ref{4} and suppose $\mathbb C(x,y)^G=\mathbb C(s)$. 
Then, taking an affine coordinate $s$, we have a morphism $p:E \longrightarrow E/G \cong \mathbb P^1$. 
Let $D$ be the polar divisor of $s$ on E.  
Next, find an element $t \in \mathbb C(x,y)$ satisfying that $\mathrm{div}(t)+D \geq 0$ and
 $\mathbb C(x,y)=\mathbb C(s,t)$. Then, it is not difficult to see 
that the curve $C$ defined by $s$ and $t$ has the Galois point at $\infty$ with the Galois group $G$. 
\end{remark}

The following examples are listed according to the structure of $G$, where we use the notation $\omega=e_3$ and $i=e_4$. 

\medskip

(I) The case where $G_T$ is trivial   

\begin{example}\label{13}
$G \cong {Z_3}$

Take the plane curve $y^2=x^3+1$. 
Let $G=\langle\sigma^*\rangle$ where $\sigma^*(x)=\omega x$, $\sigma^*(y)=y$. 
Clearly we have $\mathbb C(x,y)^G=\mathbb C(y)$ and 
 the Galois extension $\mathbb C(x,y)/\mathbb C(y)$. 
We can find two more Galois points $(0, \sqrt{-3})$ and $(0, -\sqrt{-3})$ for this curve. 
Further, this is the only smooth plane elliptic curve with a Galois point (cf. \cite{y1}).
\end{example}

\begin{example}\label{14}
$G \cong {Z_4}$

Take the plane curve $y^2=x^3+x$. 
Let $G=\langle\sigma^*\rangle$ where $\sigma^*(x)=-x$ and $\sigma^*(y)=i y$. 
Clearly we have $\mathbb C(x,y)^G=\mathbb C(x^2)$. 
Putting $s=x^2$, we have the Galois extension $\mathbb C(x,y)=\mathbb C(s,y)/\mathbb C(s)$. 
Thus we have the defining equation $y^4=s(s+1)^2$.
\end{example}

\begin{example}\label{15}
$G \cong Z_6$

Take the plane curve $y^2=x^3+1$. 
Let $G=\langle\sigma^*\rangle$ where $\sigma^*(x)=\omega x$ and $\sigma^*(y)=-y$. 
Clearly we have $\mathbb C(x,y)^G=\mathbb C(x^3)$. 
Putting $s=x^3$ and $t=xy$, we have the Galois extension $\mathbb C(x,y)=\mathbb C (s,t)/\mathbb C(s)$. 
Thus we have defining the equation $t^6=s^2(s+1)^3$.
\end{example}

Next, we find the examples where $G_T$ is not trivial. 
To represent the translation $\tau$ as the automorphism of $\mathbb C(x, y)$ we make use of the addition on $E$. 
Denote by $\star$ the addition taking the zero at the unique $\infty$. 

\medskip

(II) The case where $G_T$ has one generator

\begin{example}\label{16}
$G \cong {Z_2}^{\oplus2}$ 

In this case we take the other standard form of the elliptic curve. Let $E : y^2=x(x-1)(x-b)$ be the one, where $b \ne 0, 1$. 
Take the automorphisms $\sigma$ and $\tau$ of $E$ such that their orders are two and $\sigma$ fixes the zero and $\tau$ is a 
translation.  
The point $(b,0) \in E$ is of order two (concerning the addition on $E$) and 
we have $(x,y) \star (b,0)=(b(x-1)/(x-b), \ b(b-1)y/(x-b)^2)$.  
Hence the translation $\tau$ can be represented as $\tau^*(x)=b(x-1)/(x-b)$ and 
$\tau^*(y)=b(b-1)y/(x-b)^2)$. 
Let $G=\langle\sigma^*,\tau^*\rangle$, where $\sigma^*(x)=x$, $\sigma^*(y)=-y$. 
Clearly $x+(b(x-1))/(x-b)=(x^2-b)/(x-b)$ is invariant by $\tau^*$, so put $s=(x^2-b)/(x-b)$. 
Let $Q_1$ and $Q_2$ be the points $(b:0:1)$ and $(0:1:0)$ on $E$ respectively. Then put $D=2Q_1+2Q_2$ as a divisor. 
 It is easy to see that the polar divisor of $(x^2-b)/(x-b)$ is 
 $D$. Putting $t=(y+a)/(x-b)$, where $a \ne 0, \pm1$, we have $\mathrm{div}(t)+D \geq 0$. 
Using the relations $s=(x^2-b)/(x-b), \ t=(y+a)/(x-b)$ and $y^2=x(x-1)(x-b)$, we infer by a rather long computations  
that $\mathbb C(s,t) \ni x$ if $a \ne 0, \pm1$. Therefore we have $\mathbb C(x,y)=\mathbb C(s,t)$. 
Thus the defining equation is computed as 

\noindent $
a^4 + a^3(4b - 2s)t + ab t(-4b + 4b^2 + 2s + 2b s - 4 b^2 s -
 2s^2 + 2b s^2 - 4b t^2 + 4 b^2 t^2 + 2st^2 - 2 b st^2) + 
a^2(2b + 2b^2 - 6b s - 2b^2 s + s^2 + 4b s^2 - s^3 - 2b t^2 + 6b^2 t^2 - 4b st^2 +
s^2t^2) = b^2 (-1 + 2b - b^2 + 2s - 4b s + 2b^2 s - s^2 + 2b s^2 - b^2 s^2 - 2t^2 + 
4b t^2 - 2b^2 t^2 + 2s t^2 - 4b st^2 + 2b^2 st^2 - t^4 + 2b t^4 - b^2 t^4). $

\end{example}

\begin{example}\label{17}
$G \cong Z_2\oplus Z_4$

Take the plane curve $E : y^2=x^3+x$. 
The point $(0,0) \in E$ is of order two and 
we have $(x,y) \star (0,0)=(1/x, -y/x^2)$.  
Then the translation $\tau$ of order two can be expressed as $\tau^*(x)=1/x$ and 
$\tau^*(y)=-y/x^2$. 
Let $G=\langle\sigma^*,\tau^*\rangle$ where $\sigma^*(x)=-x$, $\sigma^*(y)=iy$, 
$\tau^*(x)=1/x$ and $\tau^*(y)=-y/x^2$.
Since $\tau^* (y/x)=-y/x$, we have $\mathbb C(x,y)^G\ni{y^4}/{x^4}$.
Therefore we obtain $\mathbb C(x,y)=\mathbb C (y^4/x^4, y)$. 
In order to check this, first put $u=y^2/x^2$, then we have $\mathbb C(u,y)\ni x^2$ 
and $\mathbb C(u,y) \ni u=y^2/x^2=(x(x^2+1))/x^2$, so we have $\mathbb C(u,y)\ni x$ and 
$\mathbb C(u,y)=\mathbb C(x,y)$. 
Further, since $x^2=y^2/u$ and $y^4=x^2(x^2+1)^2$, we have 
$u^3y^2=(y^2+u)^2$.
Then we obtain $u^2y^2=(y^2+u)^2/u=(y^4+2y^2u+u^2)/u=((y^2)^2+u^2)/u + 2y^2$.  
So we conclude $\mathbb C(u^2,y)\ni u$. Note that the polar divisor of $u^2$ is $D=4(0:0:1)+4(0:1:0)$ 
and $\mathrm{div}(y)+D \geq 0$.  
Putting $s=u^2$, we have the Galois extension $\mathbb C(x,y)=\mathbb C(y,s)/\mathbb C(s)$.
 Thus we have the defining equation $s(s-2)^2y^4=(y^4+s)^2$.
\end{example}

\begin{example}\label{18}
$G \cong {Z_3}^{\oplus2}$

Take the plane curve $E : y^2=x^3+1$. 
The point $(0,1) \in E$ is of order three and 
we have $(x,y) \star (0,1)=((2-2y)/x^2,\ (y-3)/(y+1))$.  
Hence a translation $\tau$ of order three can be expressed as $\tau^*(x)=(2-2y)/x^2$ and 
$\tau^*(y)=(y-3)/(y+1)$. 
Let $G=\langle\sigma^*,\tau^*\rangle$ where $\sigma^*(x)=\omega x$, $\sigma^*(y)=y$,
$\tau^*(x)=(2-2y)/x^2$ and $\tau^*(y)=(y-3)/(y+1)$.
 We have $\mathbb C(x,y)^{\tau^*}\ni y\tau^*(y){\tau^*}^2(y)=-y(y^2-9)/(y^2-1)$.
Putting $s=-y(y^2-9)/(y^2-1)$, we have the Galois extension $\mathbb C(x,y)=\mathbb C(x,s)/\mathbb C(s)$.
 Note that the polar divisor of $s$ is $D=3(0:1:0)+3(0:1:1)+3(0:-1:1)$ and we have $\mathrm{div}(x) +D \geq 0$. 
 Thus we have the defining equation $s^2x^6=(x^3+1)(x^3-8)^2$.
\end{example}

(III) The case where $G_T$ needs two generators. 

\begin{example}\label{19}
$G \cong {Z_2}^{\oplus3}$

Take the plane curve $y^2=x^3+x$ and two translations $\tau$ and $\tau'$ of order two. 
Let $G=\langle\sigma^*,\tau^*,{\tau'}^* \rangle$ where $\sigma^*(x)=x$, $\sigma^*(y)=-y$, 
$\tau^*(x)=1/x$, $\tau^*(y)=-y/x^2$, 
${\tau'}^*(x)=(i(x+i))/(x-i)$ and ${\tau'}^*(y)=(2y)/(x-i)^2$.
Since $x+\tau^*(x)+{\tau'}^*(x)+\tau^*{\tau'}^*(x)=(x^2-1)^2/y^2$ 
is invariant by $\tau^*, \ {\tau'}^*$ and $\sigma^*$, we have $\mathbb C(x,y)^G \ni (x^2-1)^2/y^2$.
Putting $u=(x^2-1)/y$ and $s=u^2$, we have $\mathbb C(y,u)=\mathbb C(x,y)$.
Further we have $y^4=x^2(x^2+1)^2=(uy+1)(uy+2)^2=u(sy^3+8y)+5sy^2+4$. So we have $\mathbb C(y, u)=\mathbb C(y,s)$ and 
 $(y^4-5sy^2-4)^2=sy^2(sy^2+8)^2$. 
Note that the polar divisor of $s$ is $D=2(0:1:0)+2(0:0:1)+2(i:0:1)+2(-i:0:1)$, so putting $t=1/y$, we have 
$\mathrm{div}(t)+D \geq 0$ and $\mathbb C(x,y)=\mathbb C(s,t)$.  
We obtain $(4t^4+5st^2-1)^2=st^2(s+8t^2)^2$. 
This equation gives the Galois extension $\mathbb C(s,t)/\mathbb C(s)$
\end{example}

Thus we have given the examples of genus-one curves with Galois points whose Galois group are abelian. 
Finally we raise a problem. 
We have shown in \cite{y1} (resp. \cite{y3}) the number of Galois points for a smooth plane curve (resp. genus-one curve of 
degree four) is at most three (resp. one). 
So we ask the following question.

\medskip

\noindent {\bf Question.}  Find all the Galois points for genus-one curves. In particular, is the number of 
them at most three? Then, characterize the curve with the maximal number.  

\bigskip

\noindent {\bf Acknowledgement.} 
The authors express sincere thanks for Dr. Fukasawa who gave a suggestion of the proof of Claim \ref{2}.

\bibliographystyle{amsplain}

\begin{thebibliography}{99}

\bibitem{f}S.\ Fukasawa,
Galois points for a plane curve in arbitrary characteristic,
{\em Geom. Dedicata}. {\bf 139} (2009), 211--218.  

\bibitem{s}Joseph H.\ Silverman,
The arithmetic of Elliptic Curves, 
{\em Graduate Studies in Math.} {\bf 106} (1985), Springer-Verlag 

\bibitem{m}K.\ Miura,
Galois points on singular plane quartic curves, 
{\em J. \ Algebra}, {\bf 287} (2005), 283--294.

\bibitem{my}K.\ Miura and H.\ Yoshihara,
Field theory for function fields of plane curves, 
{\em J. \ Algebra}, {\bf 226} (2000), 283--294.

\bibitem{y1}H.\ Yoshihara,
Function field theory of plane curves by dual curves,
{\em J.\ Algebra}, {\bf 239} (2001), 340--355.

\bibitem{y4}\bysame,
Galois embedding of algebraic variety and its application to abelian surface. 
{\em Rend. Sem. Mat. Univ. Padova}, {\bf 117} (2007), 69--85.

\bibitem{y2}\bysame,
Rational curve with Galois point and extendable Galois automorphism,
{\em J.\ Algebra}, {\bf 321} (2009), 1463--1472.

\bibitem{y3}\bysame,
Galois lines for normal elliptic space curves, II, 
arXiv:1004.4962v1[math.AG] {\em Algebra Colloquium} to appear.


\end{thebibliography}

\end{document}